\documentclass[11pt]{article}

\usepackage{latexsym}
\usepackage{amsmath}
\usepackage{amsthm}
\usepackage{amsfonts}
\usepackage{eufrak}
\usepackage{amssymb}
\usepackage{verbatim}
\usepackage{graphicx, color,mathrsfs}
\usepackage{amscd}

\usepackage[all]{xy}

\usepackage[numbers,sort&compress]{natbib}
\usepackage{hyperref}

\usepackage{amsfonts}

\usepackage{geometry} 
\renewcommand{\mod}[1]{{\ifmmode\text{\rm\ (mod~$#1$)}\else\discretionary{}{}{\hbox{ }}\rm(mod~$#1$)\fi}}

\def\zz{\mathbb{Z}_N}
\def\ee{\mathbb{E}}

\def\cc{\mathbb{C}}

\def\zz{\mathbb{Z}_N}
\def\zmstar{\mathbb{Z}_m^*}
\def\ee{\mathbb{E}}

\def\cc{\mathbb{C}}

\newcommand{\ktup}{b_1,\dots,b_k}
\newcommand{\rpa}{r_p(b_1,\dots,b_k)}

\newcommand{\tup}[2]{#1_1,\dots,#1_{#2}}

\newcommand{\Z}{{\mathbb Z}}
\newcommand{\lcm}{{\rm lcm}}

\newcounter{claim}
\theoremstyle{plain} 
\theoremstyle{plain} \newtheorem{proposition}{Proposition}
\theoremstyle{remark} \newtheorem{remark}[proposition]{Remark}
\theoremstyle{plain} \newtheorem{lemma}[proposition]{Lemma}
\theoremstyle{plain} \newtheorem{theorem}[proposition]{Theorem}
\theoremstyle{plain} 
\theoremstyle{plain} 
\theoremstyle{plain} 

\title{On Sums of Sets of Primes with Positive Relative Density}
\author{Karsten Chipeniuk and Mariah Hamel}

\begin{document}

\maketitle

\let\thefootnote\relax
\let\thefootnote\relax\footnotetext{The first author is supported by an NSERC Postgraduate Scholarship.  This research was supported in part by NSF VIGRE grant DMS-0738586.  Some of this work was carried out while the authors were visiting the Fields Institute.}

\begin{abstract} In this paper we show that if $A$ is a subset of the primes with positive relative density $\delta$, then $A+A$ must have positive upper density $C_1\delta e^{-C_2(\log(1/\delta))^{2/3}(\log\log(1/\delta))^{1/3}}$ in the natural numbers.  Our argument uses the techniques developed by Green and Green-Tao in their work on arithmetic progressions in the primes, in combination with a result on sums of subsets of the multiplicative subgroup of the integers modulo $m$.
\end{abstract}

\section{Introduction}

In recent years there has been much progress made toward understanding additive properties of the primes.  One of the first important structural results on the primes is due to Van der Corput \cite{CORPUT1939}, who showed that the primes contain infinitely many three term arithmetic progressions.  More recently, Green \cite{GRE2005} proved a version of Roth's theorem, by showing the existence of three term arithmetic progressions in subsets of the primes which have positive relative density.  In 2004 Green and Tao \cite{GT2008} proved the celebrated theorem that the primes contain arbitrarily long arithmetic progressions.

The strategy developed by Green and Green-Tao is to embed the primes in a `random' set where they have positive relative density and to apply a relative version of Szemeredi's theorem which holds in this setting.  Extending results from additive number theory to the setting of random sets with asymptotic density $0$ was first considered by Kohayakawa-{\L}uczak-R\"{o}dl \cite{KLR} who proved a variant of Roth's theorem.  An alternate proof of this version of Roth's theorem is proved in Tao-Vu \cite{TV2006} and lends itself to adaptation in the primes (similar to Green's proof of Roth's theorem in the primes and used recently in \cite{Thai2009}).  This method of embedding the primes in a `random' set suggests that one should be able to prove other results which are known in a random setting to that of the primes.  A result of \L aba and the second author \cite{HL2008} says that if $A$ is a subset of a random set in $\zz:=\mathbb{Z}/N\mathbb{Z}$ with positive relative density, then $A+A$ must have positive density in $\zz$.

\begin{theorem}\label{sumsets-size}
Suppose that $S$ is a random subset of $\zz$ such that the events
$x\in S$, where $x$ ranges over $\zz$, are independent and have
probability $p=p(N)\in(CN^{-\theta},1]$, where $0<\theta<1/140$. Then
for every $\beta<\alpha$, the statement
\begin{quote}
for every set $A\subset S$ with $|A|\geq \alpha |S|$,
we have $|A+A|\geq \beta N$
\end{quote}
is true with probability $1-o_{\alpha,\beta}(1)$ as $N\to\infty$.

\end{theorem}

The main result of this paper is a version of Theorem \ref{sumsets-size} where $A$ is replaced by a relatively dense subset of the primes, which take the role of the random set $S$.  We should note that it is known that if $\mathcal{P}$ is the set of primes then the density of $\mathcal{P}+\mathcal{P}$ in the natural numbers is $1/2$ (see, for example, \cite{V1997}).

\begin{theorem}\label{sumsetsinP}  Let $A$ be a subset of the primes with positive relative density $0<\delta_0<1$.  Then there exist absolute constants $C_1$ and $C_2$ such that $A+A$ has positive upper density at least 
\begin{equation}\label{density}C_1\delta_0 e^{-C_2(\log(1/\delta_0))^{2/3}(\log\log(1/\delta_0))^{1/3}}
\end{equation}
in the natural numbers.

\end{theorem}

\begin{remark} For $\delta_0=1$, a modification of our argument can be used to show that $A+A$ has upper density 1/2 (in particular, one can use the Chinese remainder theorem to prove a corresponding version of Theorem \ref{ZNstar} below).  For sufficiently small values of $\delta_0$,  the constant $C_1$ can be absorbed into the exponential term.  For larger values of $\delta_0$ we can use the boundedness of the argument of the exponential near $\delta_0=1$ to rewrite the density as $C_1\delta_0$ for a new value of $C_1$.
\end{remark}

While we do not believe that this bound is best possible, the following example shows that it is not possible to extend Theorem \ref{sumsetsinP} to the the analogous conclusion of Theorem \ref{sumsets-size}.

Let $\varphi$ denote the Euler totient function, that is, for an integer $n$, $\varphi(n)$ is the number of integers less than $n$ which are relatively prime to $n$.  Let $m$ be the product of the first $t$ primes and define $$A:=\{p\in\mathcal{P}:p\equiv 1(\text{mod } m)\}.$$  Let $A_n$ denote the set of elements of $A$ which are less than or equal to $n$.  The prime number theorem for arithmetic progressions implies that $$|A_n|=\frac{n}{\varphi(m)\log n}+O\big(\frac{n}{\log^2 n}\big).$$  Hence, if $n$ is sufficiently large, we have $$|A_n|\geq\frac{n}{2\varphi(m)\log n}.$$  It follows by the prime number theorem that the relative density of $A$ in the set of primes is at least $\delta:=1/2\varphi(m)$.  On the other hand, the definition of $A$ implies that $$A_n+A_n\subset\{s\in\mathbb{N}:s\equiv 2 (\text{mod } m)\}.$$  Using estimates of Chebyshev and Mertens (for example, see Theorem 8.2 and Theorem 8.8 in \cite{Nathanson}) we see that $m\sim \varphi(m)\log\log\varphi(m),$ and hence $$|A_n+A_n|\leq \frac{2n}{m}\sim \frac{\delta}{\log\log (1/\delta)}n.$$

We also remark that it is possible to replace the bound in (\ref{density}) with the weaker bound of $\delta^2$ using a much simpler argument which uses Cauchy-Schwarz and the Brun sieve.

One important difference between proving Theorem \ref{sumsets-size} and Theorem \ref{sumsetsinP} is that while $S$ is defined to be a random set in Theorem \ref{sumsets-size}, the set of primes is not randomly distributed.  For example, there is only one prime which is divisible by $2$, and if $x\neq 2$ is prime then the probability that $x+1$ is prime is zero.  The strategy employed by Green and Green-Tao to handle this difficulty is to consider the primes modulo $m$ where $m$ is the product of small primes.  They then pick one residue class where $A\subset\mathcal{P}$ has large density and find an arithmetic progression contained in that residue class.  

In order to bound the density of $A+A$ we are not able to restrict the arguments to one residue class.  To prove Theorem \ref{sumsetsinP} we will need a way to consider all residue classes for which $A$ has large relative density.  The relevant residue classes are contained in the multiplicative subgroup of the integers modulo $N$, which we will denote by $\zz^{\ast}$. The result that we need is contained in the proof of the following theorem.

\begin{theorem}\label{ZNstar}
Let $0<\alpha < 1$.  Assume that $m\in\mathbb{Z}^+$ is sufficiently large depending on $\alpha$.  If $B\subset \mathbb{N}$ satisfies $|B|\geq \alpha \varphi(m)$ then there are absolute constants $C_1$ and $C_2$ such that
$$|B+B|\geq C_1\alpha e^{-C_2(\log(1/\alpha))^{2/3}(\log\log(1/\alpha))^{1/3}}m.$$ 
\end{theorem}

It is not a coincidence that the conclusions of Theorems \ref{sumsetsinP} and \ref{ZNstar} contain the same exponential factor.  In fact, it will be evident in the proof of Theorem \ref{sumsetsinP} that this factor comes about entirely from the structure of $\mathbb{Z}_m^\ast$ for a suitable modulus $m$.  Furthermore, the following example, in conjunction with Freiman's theorem (see, for example, Theorem 5.33 in \cite{TV2006}), suggests that the $1/\log\log(1/\delta)$ factor obtained in the previous construction may in fact be sharp:

Suppose $m=p_1\cdots p_s$, where $p_1<\cdots<p_s$ are the first $s$ primes.  Then $\mathbb{Z}_m\cong\mathbb{Z}_{p_1}\times\cdots\times\mathbb{Z}_{p_s}$.  Let $B\subset\mathbb{Z}_m$ be given in this representation by
$$B:=\{1\}\times \cdots\times\{1\}\times(\mathbb{Z}_{p_{t+1}}\setminus\{0\})\times \cdots\times(\mathbb{Z}_{p_s}\setminus\{0\}).$$
Here $B$ is a low dimensional generalized arithmetic progression.  Then, using the notation of Theorem \ref{ZNstar}, we can directly compute
$$\alpha = 1/\varphi(p_1\dots p_t)$$
from the fact that $\varphi(p)=p-1$ for a prime $p$ and the fact that $\varphi$ is multiplicative.  In addition, computing $B+B$ and using the same estimates as in the previous example, we obtain
$$|B+B| = \frac{m}{p_1\dots p_t} \sim \frac{\alpha}{\log\log(1/\alpha)} m.$$


\section{Preliminaries and an outline of the argument}


Throughout this paper $A$ will be a subset of the primes, $A_n$ will be a subset of the primes which are less than or equal to $n$ and $A_n^{(b)}$ will be those elements of $A_n$ which are congruent to $b$ modulo $m$, where $m$ is the product of the primes less than or equal to some sufficiently large parameter $W$.  We use $|A|$ to denote the cardinality of the set $A$ and define $A+A:=\{a+a':a,a'\in A\}$.  We write $C$, $C_1$, or $C_2$ to denote an absolute constant, although the exact value of any of these may differ between any two different expressions.  For real-valued functions $f$ and $g$, we write $f\ll g$ to mean $|f|\leq C |g|$.  As previously noted, we write $\zz$ to denote the cyclic group $\mathbb{Z}/N\mathbb{Z}$ and $\zz^*:=\{x\in\zz:(x,N)=1\}$ to denote the multiplicative subgroup of integers modulo $N$. 

If $f:\zz\rightarrow\cc$ then we define the expectation of $f$ to be 
$$\ee(f):=\frac{1}{N}\sum_{x\in\zz}f(x).$$  
We define the normalized Fourier transform 
$$\widehat{f}(\xi):=\frac{1}{N}\sum_{x\in\zz}f(x)e(-x\xi/N)$$ 
where $e(\alpha):=\exp(2\pi i\alpha)$.  For two functions $f,g:\zz\rightarrow\cc$ we define the convolution 
$$f*g(x):=\sum_{y\in\zz}f(y)g(x-y).$$  
We also define the $L^p$ norm $$\|f\|_p:=\big(\sum_{x\in\zz}|f(x)|^p\big)^{1/p}$$ 
and the $L^{\infty}$ norm 
$$\|f\|_{\infty}:=\sup_{x\in\zz}|f(x)|.$$ 
We will use Plancherel's identity which says that 
$$\sum_{\xi\in\mathbb{Z}_N}\widehat{f}(\xi)\;\overline{\widehat{g}(\xi)}=N^{-1}\sum_{x\in\mathbb{Z}_N} f(x)\overline{g(x)},$$ 
an identity for convolution 
$$\widehat{f\ast g}(\xi)=N\widehat{f}(\xi)\widehat{g}(\xi)$$ 
and the Fourier inversion formula 
$$f(x)=\sum_{\xi\in\zz}\widehat{f}(\xi)  e(x\xi/N).$$
We will say that a function $f:\zz\rightarrow\mathbb{R}_{\geq 0}$ is \emph{pseudorandom} if $$\|\widehat{f}(\xi)-1_{\xi=0}\|_\infty\leq \eta$$ for some $0<\eta\leq1$.

\bigskip

An outline of our argument is as follows: In Section 3, we begin by partitioning $A$ into residue classes modulo $m$, where $m$ is the product of small primes.  We then use techniques introduced in \cite{GRE2005} to embed each residue class on which $A$ is concentrated into $\mathbb{Z}_N$ for $N\sim n/m$.  In this setting, we are able to utilize the concept of pseudorandomness to decompose a modified characteristic function of $A$ on each partition (simultaneously) into a bounded part and a linearly uniform part.  Modifying the arguments used to prove Theorem \ref{sumsets-size}, we show that the sumset of the images of any two congruence classes of $A$  has comparably large density in $\mathbb{Z}_N$.

In Section 4, we develop a moment estimate needed to prove Theorem \ref{ZNstar}.  In particular, in Proposition \ref{representations}, we prove a $k$th moment estimate of the representation function which bounds the number of ways to write an element of $\mathbb{Z}_m$ as the sum of two elements in $B$.  An application of H\"older's inequality allows us to use the $k$th moment estimate to prove Theorem \ref{ZNstar}.

In Section 5, we combine the results of Sections 3 and 4 with an application of H\"older's inequality to complete the proof of a finite version of Theorem \ref{sumsetsinP}.


\section{Sumsets and uniformity of the primes in residue classes}


The main goal of this section will be to prove Proposition \ref{sumset1} below.  Before stating Proposition \ref{sumset1} we will state a finite version of Theorem \ref{sumsetsinP} which will allow us to introduce necessary notation.

Let $0<\delta_0 < 1$.  Let $A$ be a subset of the primes with positive relative density $\delta_0$.  This means that $$\limsup_{n\rightarrow\infty}\frac{|A\cap\mathcal{P}_n|}{|\mathcal{P}_n|}=\delta_0,$$ and hence there exist infinitely many $n$ so that $\frac{|A\cap\mathcal{P}_n|}{|\mathcal{P}_n|}\geq \delta_0/2$.  Theorem \ref{sumsetsinP} will then follow from a finite version, where $\delta:=\delta_0/2$:

\begin{theorem}\label{sumsetsinPN}
Let $A_n\subset \mathcal{P}_n$ satisfy $|A_n|\geq \delta |\mathcal{P}_n|$.  Then there exist absolute constants $C_1$ and $C_2$ such that if $n\geq n_0(\delta)$ then
$$|A_n+A_n| \geq C_1\delta e^{-C_2(\log(1/\delta))^{2/3}(\log\log(1/\delta))^{1/3}} n.$$
\end{theorem}

Let $\epsilon>0$ be a small parameter and let $W$ be sufficiently large depending on $\delta$ and $\epsilon$, and satisfying $W\ll \log\log n$.  Set
$$m=\prod_{p\leq W} p.$$
We begin by partitioning $A_n$ into congruence classes modulo $m$.  More specifically, let
$$A_n^{(b)} = \left\{a\in A_n: a\equiv b\ \mod{m}\right\},$$
and
$$\mathcal{P}_n^{(b)} = \left\{p\in \mathcal{P}_n: p\equiv b\ \mod{m}\right\}.$$
Define
$$\delta_b = \frac{|A_n^{(b)}|}{|\mathcal{P}_n^{(b)}|}.$$

For those sets $A_n^{(b)}$ for which we have a large relative density in $\mathcal{P}_n$, we say that $b$ is good, and we define the set of good residue classes to be $$G=\{b\in\zmstar:\delta_b\geq\delta/2\}.$$

Combining the methods of Green \cite{GRE2005} on three term arithmetic progressions in subsets of the primes with the methods used to prove Theorem \ref{sumsets-size} we are able to show that for any pair of good residue classes $b_1$ and $b_2$ the sumset $A_n^{(b_1)}+A_n^{(b_2)}$ is dense in the progression $\{0\leq x\leq 2n: x\equiv b_1+b_2\ \mod{m}\}$.  Summing over pairs of residue classes and being careful not to count multiplicities, we have the following:

\begin{proposition}\label{sumset1}
For every $x\in G+G$, let
$$\Delta_x = \max_{(b,b')\in G\times G \atop b+b' = x}\left(\frac{\delta_b+\delta_{b'}}{2}\right)\geq \delta/2.$$
Then for every $\epsilon>0$
\begin{equation}
|A_n + A_n| \geq \sum_{x \in G+G} \left(\Delta_x-\epsilon \right) \frac{n}{m}.
\end{equation}
\end{proposition}

\begin{remark}
As we will see in the proof of Lemma \ref{decomposition}, we will require $m$ to be a rapidly increasing function as $\delta$ and $\epsilon$ go to $0$.  This is the reason that we must gain control on $|G+G|$ in order to prove Theorem \ref{sumsetsinPN}.
\end{remark}

In the remainder of this section, we will prove Proposition \ref{sumset1}.  This requires several lemmas.  

The first lemma, which was proved by Green (\cite{GRE2005}, Lemma 6.1), allows us to consider $A_n^{(b)}$ as a subset of $\zz$ for $N\sim n/m$.  For those $b$ for which the density $\delta_b$ is large, we will then construct a collection of simultaneously pseudorandom measures which majorize modified characteristic functions of the images in $\mathbb{Z}_N$ of the sets $A_n^{(b)}$.  Working in $\zz$ we then use Fourier analytic techniques.  As our notation differs slightly from that used by Green in \cite{GRE2005}, we include a proof for completeness.  

We recall Green's modified von Mangoldt function $\lambda_{b,m,N}:\mathbb{Z}^+\to \mathbb{R}$, defined as
$$\lambda_{b,m,N}(x)=\left\{ \begin{array}{ll}
\frac{\varphi(m)}{mN}\log(mx+b) & \textrm{if $x\leq N$ and $mx+b$ is prime,}\\
0 & {\rm otherwise.}
\end{array} \right.
$$

\begin{lemma}\label{embedding}
Let $N\in(2n/m, 4n/m]$, and let $\mathcal{A}_N^{(b)} = \big(m^{-1}(A_n - b)\big)\cap \{1,\dots,N\}$ for $b\in G$.  Then
\begin{equation*}
\sum_{x\in \mathcal{A}_N^{(b)}} \lambda_{b,m,N} (x) \geq \frac{\delta_b}{16}.
\end{equation*}
\end{lemma}

\begin{proof}
By the prime number theorem we have

\begin{eqnarray*}
\sum_{x\in A^{(b)}_n} \log x & \geq & \sum_{x\geq n^{3/4}} 1_{A_n^{(b)}}(x) \log x \nonumber\\
& \geq &
\left( \frac{\delta_b n}{\varphi(m) \log n} - n^{3/4}\right) \log (n^{3/4})\nonumber\\
& = &
\frac{3}{4} \frac{\delta_b n}{\varphi(m)} - \frac{3}{4} n^{3/4} \log n \nonumber\\
& \geq &
\frac{\delta_b n}{4\varphi(m)}
\end{eqnarray*}
for $n$ sufficiently large.  Performing a change of variables $x\to mx+b$ we obtain
$$\sum_{x\leq N \atop mx+b\ \textrm{prime}} 1_{\mathcal{A}_N^{(b)}}(x) \log(mx+b) \geq \frac{\delta_b n}{4\varphi(m)}.$$
By definition of $N$,
\[\sum_{x\in \mathcal{A}_N^{(b)}} \lambda_{b,m,N}(x) \geq \delta_b n/4mN \geq \delta_b/16. \qedhere\]
\end{proof}

Taking $\mathcal{A}_N^{(b)}\subset\zz$ we notice that $|\mathcal{A}_N^{(b)}|=|A_n^{(b)}|$ and \begin{equation}\label{equationzn}|A_n^{(b_1)}+A_n^{(b_2)}|\geq |\mathcal{A}_N^{(b_1)}+\mathcal{A}_N^{(b_2)}|\end{equation} for any $b_1,b_2\in G$.

\bigskip

We now define $$f^{(b)}(x) := N1_{\mathcal{A}_N^{(b)}}(x) \lambda_{b,m,N}(x),$$ and $$\nu^{(b)}(x) := N\lambda_{b,m,N}(x).$$  From the above lemma, we note that $\mathbb{E}f^{(b)} \geq \delta_b /16$.  Such functions were first defined by Green in \cite{GRE2005} where a three term arithmetic progression was located in a fixed residue class $b$ where $A_n^{(b)}$ had large relative density.  The function $f$ is defined so that it has large expectation and $\nu$ is \emph{pseudorandom}.  We require the following two lemmas of Green (\cite{GRE2005}, Lemma 6.2 and Lemma 6.6)  which express the pseudorandom properties of primes:

\begin{lemma}\label{pseudorandom}
For $N$ and $W$ sufficiently large there is some $D>0$ such that
\[\widehat{\nu^{(b)}} (0)  \leq 1+ O((\log N)^{-D})\] and
\[\sup_{\xi\neq 0} |\widehat{\nu^{(b)}} (\xi)|  \leq 2\log\log W /W.\]
\end{lemma}

\begin{lemma}\label{discmaj}
Let $s>2$.  Then there is a constant $C(s)$ such that

\begin{equation}
\|\widehat{f^{(b)}}\|_s \leq C(s).
\end{equation}
\end{lemma}

In order to prove Proposition \ref{sumset1} we will show that if $b_1, b_2\in G$ then \begin{equation}\label{fb1fb2}|\{x\in\zz:(f^{(b_1)}\ast f^{(b_2)})(x)>0\}|\geq\big(\frac{\delta_{b_1}+\delta_{b_2}}{2}-\epsilon\big)N.\end{equation}  Since $f^{(b_1)}\ast f^{(b_2)}$ is supported on $\mathcal{A}_N^{(b_1)}+\mathcal{A}_N^{(b_2)}$, assuming (\ref{fb1fb2}) implies that the size of this sumset must be large.  In this case, by (\ref{equationzn}), we must have
\begin{eqnarray}
|A_n^{(b_1)}+A_n^{(b_2)}| & \geq & \big(\frac{\delta_{b_1}+\delta_{b_2}}{2}-\epsilon\big)N\nonumber\\
& \geq & \big(\frac{\delta_{b_1}+\delta_{b_2}}{2}-\epsilon\big)\frac{n}{m}.\nonumber
\end{eqnarray}
Noticing that $A_n^{(b_1)}+A_n^{(b_2)}$ is disjoint from $A_n^{(b_3)}+A_n^{(b_4)}$ provided $b_1+b_2 \neq b_3+b_4$ proves Proposition \ref{sumset1}.  It is therefore sufficient to prove (\ref{fb1fb2}).

Equation (\ref{fb1fb2}) follows, with modifications, from the arguments in \cite{HL2008}.  These arguments rely on a Fourier-analytic decomposition of Green \cite{GRE2005} and Green-Tao \cite{GT2008}, which as stated, appears in \cite{GT2006} [see Proposition 5.1] and is also contained in \cite{TV2006} [see Theorem 10.20].  In particular the functions $f^{(b)}$ are decomposed as $f_1^{(b)}+f_2^{(b)}$ where $f_1^{(b)}$ is bounded and $f_2^{(b)}$ is unbounded but `uniform'.

\begin{lemma}\label{decomposition}  Suppose that $f=f^{(b)}$ and $\nu=\nu^{(b)}$ are as above.  Let $s>2$ and let $\epsilon_0 > 0$.  Define $$f_1(x):=\ee(f(x+y_1-y_2):y_1,y_2\in B_0),$$  where $$B_0:=\{x:|e^{-2\pi ix\xi/N}-1|\leq \epsilon_0\ {\rm for\ all\ }\xi\in\Lambda_0\} ,\ \ \ \ \Lambda_0:=\{\xi:|\widehat{f}(\xi)|\geq\epsilon_0\}.$$  Define $f_2(x):=f(x)-f_1(x)$.  Then for every $\sigma>0$, assuming that $N$ is sufficiently large, and $W$ is sufficiently large depending on $\epsilon_0$, we have

(i) $0\leq f_1\leq 1+\sigma,$

(ii) $\ee f_1=\ee f$,

(iii) $\|\widehat{f_2}\|_{\infty}\leq\epsilon_0/16$ and $\|\widehat{f_1}\|_{\infty}\ll 1$,

(iv) $\|\widehat{f_i}\|_s\ll 1$ for $i=1,2$.

\end{lemma}

\begin{proof} The proofs of (ii), (iii) and (iv) follow as in \cite{GT2006}, while we reiterate the proof of (i) here.  

In order to bound $f_1$ we begin by using Fourier inversion to show that  
\begin{eqnarray*}
0\ \leq\ f_1(x) & = & \frac{1}{|B_0|^2}\sum_{y_1,y_2\in B_0} f(x+y_1-y_2)\\
& \leq & \frac{1}{|B_0|^2}\sum_{y_1,y_2\in B_0} \nu(x+y_1-y_2)\\
& = & \frac{1}{|B_0|^2}\sum_{y_1,y_2\in B_0}\sum_{\xi\in\mathbb{Z}_N} \widehat{\nu}(\xi)e\left(\frac{\xi(x+y_1-y_2)}{N}\right)\\
& = & \sum_{\xi\in\mathbb{Z}_N}\widehat{\nu}(\xi)e\left(\xi x/N\right)\frac{1}{|B_0|^2}\left|\sum_{y\in B_0} e\left(-\xi y/N\right)\right|^2\\
& \leq & \sum_{\xi\in\mathbb{Z}_N}|\widehat{\nu}(\xi)|\frac{1}{|B_0|^2}\left|\sum_{y\in B_0} e\left(-\xi y/N\right)\right|^2.\\
\end{eqnarray*}
We continue by applying Lemma \ref{pseudorandom} and Plancherel's identity to show that is
\begin{eqnarray*}
& \leq & |\widehat{\nu}(0)|+\sum_{\xi\in\mathbb{Z}_N} \frac{2\log\log W}{W}\frac{1}{|B_0|^2}\left|\sum_{y\in B_0} e\left(-\xi y/N\right)\right|^2\\
& \leq & 1+O\left((\log N)^{-D}\right) + \frac{2\log\log W}{W}\frac{N^2}{|B_0|^2}\sum_{\xi\in\mathbb{Z}_N} \left| \widehat{1_{B_0}}(\xi)\right|^2\\
& = & 1+O\left((\log N)^{-D}\right) + \frac{2\log\log W}{W}\frac{N}{|B_0|}\\
& = & 1+ O\left(\frac{2\log\log W}{W}\frac{N}{|B_0|}\right).
\end{eqnarray*}
Using the pigeonhole principle, there is a constant $c>0$ so that
$$|B_0|\geq (c\epsilon_0)^{|\Lambda_0|} N.$$
Also,
$$\sum_{\xi\not\in \Lambda_0} |\widehat{f}(\xi)|^s + \epsilon_0^s|\Lambda_0| \leq C(s)^s$$
where $C(s)$ is given by Lemma \ref{discmaj}, so
$$|\Lambda_0|\leq (C(s)/\epsilon_0)^s.$$
Therefore
$$0\leq f_1(s) \leq 1+ O\left(\frac{\log\log W}{W(c\epsilon_0)^{\left(C(s)/\epsilon_0\right)^s}}\right).$$
The bound now follows since $W$ is sufficiently large in terms of $\epsilon_0$.
\end{proof}

\begin{remark} In the following we will see that we must take $\epsilon_0$ to be smaller than $\delta^{4}\epsilon^{6}$.  Combining this with the fact that we need the error term in the final equation of the above proof to be bounded, we see that $m$ must increase rapidly as $\delta$ and $\epsilon$ approach 0, as mentioned in the remark following Proposition \ref{sumset1}.
\end{remark}

\begin{lemma}\label{convolution}  Suppose that $f, g:\zz\rightarrow\mathbb{C}$ are functions so that $$\ee(f)=\alpha,$$ $$\ee(g)=\beta$$ and which have the property that they are majorized (respectively) by pseudorandom functions $\nu, \mu:\zz\rightarrow\mathbb{C}$, that is, $$0\leq f(x)\leq \nu(x)$$ and $$0\leq g(x)\leq\mu(x)$$ for every $x\in\zz$.   Then for every $\epsilon>0$ \begin{equation}\label{convolution1}|\{x\in\zz:(f\ast g)(x)>0\}|\geq\big(\frac{\alpha+\beta}{2}-\epsilon\big)N.\end{equation}
\end{lemma}

\begin{proof} Without loss of generality, assume that $0<\alpha\leq \beta$ and let $\sigma$ be a parameter which satisfies $0<\sigma<\epsilon/10$.  We decompose $f=f_1+f_2$ and $g=g_1+g_2$ as in Lemma \ref{decomposition}, with $\epsilon_0$ depending on $\alpha$ and $\sigma$ to be chosen later.  

In order to establish (\ref{convolution1}), it suffices to prove a main term estimate 
\begin{equation}\label{convolutionmain} |\{x\in\zz: f_1\ast g_1(x)>\sigma\alpha N\}|\geq\big(\frac{\alpha+\beta}{2}-3\sigma\big)N\end{equation}
and three error terms of the form
\begin{equation}\label{convolutionerror}|\{x\in\zz:|f_i\ast g_j(x)|>\frac{\sigma\alpha}{10} N\}|\leq\sigma N,\end{equation} where $(i,j)\neq (1,1)$.

For the main term, we first notice that since $f$ and $g$ are both nonnegative 

\begin{equation}\label{fg11}\|f_1\ast g_1\|_1=\|f_1\|_1\|g_1\|_1=\alpha\beta N^2.\end{equation}
If (\ref{convolutionmain}) were false, then we would have $$\|f_1\ast g_1\|_1\leq \sigma\alpha N^2+\alpha(1+\sigma)(\frac{\alpha+\beta}{2}-3\sigma)N^2\leq\alpha\beta N^2-\alpha\sigma N^2$$ which contradicts (\ref{fg11}).

For the error terms, we will show the argument for $j=2$ (the other estimate follows similarly).  It is sufficient to show that

\begin{equation}\label{fstarg2}\|f_i\ast g_2\|_2^2\leq\sigma N\left(\frac{\sigma^2\alpha^2}{200} N^2\right).\end{equation}  Using the convolution identity for Fourier transforms and the Cauchy-Schwarz inequality, we have \begin{align*}\|f_i\ast g_2\|_2^2&=\sum_{x\in\zz}|(f_i\ast g_2)(x)|^2\\
	&=N\sum_{\xi\in\zz}|\widehat{f_i\ast g_2}(\xi)|^2\\
	&=N^3\sum_{\xi\in\zz}|\widehat{f_i}(\xi)|^2|\widehat{g_2}(\xi)|^2\\
	&\leq N^3\|\widehat{f_i}\|_{\infty}^{1/2}\|\widehat{g_2}\|_{\infty}^{1/2}\sum_{\xi\in\zz}|\widehat{f_i}(\xi)|^{3/2}|\widehat{g_2}(\xi)|^{3/2}\\
	&=N^3\|\widehat{f_i}\|_{\infty}^{1/2}\|\widehat{g_2}\|_{\infty}^{1/2}\|\widehat{f_i}\|_{3}^{3/2}\|\widehat{g_2}\|_{3}^{3/2}\\
	&\ll N^3 \epsilon_0^{1/2}.
\end{align*}

Hence, to ensure (\ref{fstarg2}) we simply require $\epsilon_0^{1/2}\leq \sigma^3\alpha^2/200$, or equivalently $\epsilon_0\leq\frac{\sigma^{6}\alpha^{4}}{2^{4}5^{2}}$.
\end{proof}


\section{Sumsets of Positive Density Subsets of $\mathbb{Z}_m^\ast$}


Let $m\in \mathbb{Z}^+$.  For $B\subset \mathbb{Z}_m$ and $x\in \mathbb{Z}_m$, denote
$$r_B(x) = |\{(b,b')\in B\times B: b+b'=x\}|.$$

In this section our main objective is to prove the following:

\begin{proposition}\label{representations}Let $\alpha>0$.  Suppose $m\in \mathbb{Z}^+$ is squarefree, and let $B\subset \mathbb{Z}_m^\ast$ satisfy $|B|\geq \alpha \varphi(m)$.  Then there exists an absolute constant $C$ such that if $m\geq m_0=m_0(\alpha)$ and $k\in\mathbb{Z}^+$ then
$$\sum_{x\in \mathbb{Z}_m} r_B(x)^k \leq \left(\frac{e^{C k^3\log(k)}}{\alpha^2}\right) \frac{|B|^k\varphi(m)^k}{m^{k-1}}.$$
\end{proposition}

As a corollary, we obtain Theorem \ref{ZNstar} using H\"older's inequality.

\bigskip

\noindent
\emph{Proof of Theorem \ref{ZNstar}}: Assume that $B$ is a set which satisfies the hypotheses of Theorem \ref{ZNstar}.  We first assume that $m$ is square free and at the end of the proof, we reduce the general case to the square free one. 

Using H\"older's inequality, for any $k\in \mathbb{Z}^+$, we have
$$\sum_{x\in \mathbb{Z}_m} r_B(x)  = \sum_{x\in B+B} r_B(x) \leq  \left( \sum_{x\in \mathbb{Z}_m} r_B(x)^k\right)^{1/k} \left(\sum_{x\in B+B} 1\right)^{(k-1)/k}.$$
The sum on the left is just $|B|^2$.  Hence, Proposition \ref{representations} implies that
\begin{eqnarray}
|B+B| & \geq & \frac{|B|^{2k/(k-1)}}{\left(\sum_{x\in \mathbb{Z}_m} r_B(x)^k\right)^{1/(k-1)}}\nonumber\\
& \geq & \frac{|B|^{2k/(k-1)}}{(|B|\varphi(m))^{k/(k-1)}} \left(\alpha^{2/(k-1)} e^{-2C(k-1)^2\log(k-1)}\right) m\nonumber\\
& = & \alpha^{1+\frac{3}{k-1}}e^{-2Ck^3\log(k)/(k-1)}m\nonumber\\
& = & \alpha^{1+\frac{6}{k}}e^{-2Ck^2\log(k)}m\nonumber
\end{eqnarray}
for $k>2$.  Taking $k=\lfloor \big(\frac{\log(1/\alpha)}{\log\log(1/\alpha)}\big)^{1/3}\rfloor$ we find
$$|B+B| \geq \alpha e^{-C_2(\log(1/\alpha))^{2/3}(\log\log(1/\alpha))^{1/3}} m$$
for $\alpha$ sufficiently small.  

To deal with the case when $\alpha$ is not small, suppose $\alpha_0$ is the largest density for which we know the theorem is true.  Partition $B$ into $B_1, \dots, B_{\lfloor\alpha/\alpha_0\rfloor}$ sets each of which contains exactly $\alpha_0 \varphi(m)$ consecutive elements of $B$.  We apply the known result to each set $B_j$ for $1\leq j\leq \lfloor\alpha/\alpha_0\rfloor$ to obtain $$|B_j+B_j|\geq \alpha_0 e^{-C_2(\log(1/\alpha_0))^{2/3}(\log\log(1/\alpha_0))^{1/3}} m.$$  Summing over all $j$, we have 
\begin{align*}
	|B+B|&\geq \sum_{j=1}^{\lfloor\alpha/\alpha_0\rfloor}|B_j+B_j|\\
	&\geq  \sum_{j=1}^{\lfloor\alpha/\alpha_0\rfloor}\alpha_0 e^{-C_2(\log(1/\alpha_0))^{2/3}(\log\log(1/\alpha_0))^{1/3}} m\\
	&\geq C_1\alpha e^{-C_2(\log(1/\alpha))^{2/3}(\log\log(1/\alpha))^{1/3}} m
\end{align*}
as desired.  Note that the constant $C_1= e^{-C_2(\log(1/\alpha_0))^{2/3}(\log\log(1/\alpha_0))^{1/3}}$ where $\alpha_0$ is the largest $\alpha$ for which we know the result is true. This completes the proof when $m$ is squarefree.

We reduce the general case to the squarefree one by letting
$$m_1=\prod_{p|m} p,$$
and considering the intervals $I_j=[jm_1,(j+1)m_1)$ for $j=0,\dots, m/m_1-1$.
Let $B_j=B\cap I_j$, denote $\alpha_j=\frac{|B_j|}{m_1}$, and let $J=\{j: \alpha_j>\alpha/2\}$.  Notice that 
$$\sum_{j=1}^{m/m_1-1}\alpha_j=\alpha \frac{m}{m_1}$$ 
and so 
$$\sum_{j\in J}\alpha_j\geq \frac{\alpha}{2}\frac{m}{m_1}.$$
Considering  $I_j$ and $B_j$ as subsets of integers, we note that $(B_j+B_j)\cap (B_i+B_i)=\emptyset$ for $i\neq j$.  For each $j\in J$ we apply the theorem to the translate $B_j - jm_1$ to get $$|B_j+B_j| \geq C_1\alpha_j e^{-C_2(\log(1/\alpha_j))^{2/3}(\log\log(1/\alpha_j))^{1/3}} m.$$  Therefore, as subsets of integers,
\begin{align*}
	|B+B|&\geq \sum_{j\in J}|B_j+B_j|\\
	&\geq  \sum_{j\in J}C_1\alpha_j e^{-C_2(\log(1/\alpha_j))^{2/3}(\log\log(1/\alpha_j))^{1/3}} m_1\\
	&\geq C_1 e^{-C_2(\log(1/\alpha))^{2/3}(\log\log(1/\alpha))^{1/3}}m_1\sum_{j\in J}\alpha_j\\
	&\geq C_1\alpha e^{-C_2(\log(1/\alpha))^{2/3}(\log\log(1/\alpha))^{1/3}} m
\end{align*}
as desired.\qed

\bigskip

\emph{Proof of Proposition \ref{representations}:}.  Let $B$ be a set satisfying the hypotheses of the lemma, and suppose that the prime factorization of $m$ is $m=p_1\dots p_t$. 
Define
\begin{align*}
R(x):=&|\{(b,r)\in B\times\mathbb{Z}_m^\ast: b+r=x\}| \\
=&|\{b\in B: b\not\equiv x\ \mod{p_1},\dots,b\not\equiv x\ \mod{p_t} \}|.
\end{align*}

Then $R(x)$ is simply counting the representations of $x$ as a sum involving an element in $B$ and an element taken from the whole of $\mathbb{Z}_m^\ast$, and in particular $R(x)\geq r_B(x)$.  Although $R(x)$ is larger than the function $r_B(x)$, it is easier to control.  Our goal is to produce a good upper bound on the $k$th moment of $R(x)$.  Define
$$S := \sum_{x\in\mathbb{Z}_m} R(x)^k.$$

We begin by separating the values of $x$ into partitions based on the value of $(x,m)$.  More specifically, for $d|m$ let
\begin{align*}
X_d:= & \{x\in [0,m-1]: (x,m)=d\}\\
= & \{x\in [0,m-1]: x=dl\ {\rm for\ some\ }l\in [0,m/d -1]\ {\rm with\ } (l,m/d)=1\}.
\end{align*}
Then we have
$$S= \sum_{d|m} \sum_{x\in X_d} R(x)^k = \sum_{d|m} \sum_{x\in X_d} |\{b\in B:  b\not\equiv x\ \mod{p}\text{ for all }\ p|m/d\}|^k$$
since the conditions $b\not\equiv x\ \mod{p}$ for $p|d$ are certainly satisfied by every $b\in B$ by the condition $(b,m)=1$.

Now, we denote the inner sum by $S_d$, so that
$$S\leq \sum_{d|m} S_d.$$
Expanding the $k$th power in $S_d$ and rearranging the order of summation gives
$$S_d \leq \sum_{\ktup\in B} \sum_{x\in X_d \atop \ktup\not\equiv x\ \mod{p}\text{ for all }\ p|m/d} 1.$$
Now, fix a $k$-tuple $\ktup\in B$.  For this $k$-tuple, the contribution of the inner sum above is
\begin{equation}\label{needaname}
|X_d\cap\{x\in[0,m-1]: x\not\equiv \ktup\ \mod{p}\text{ for all }\ p|m/d\}|
\end{equation}
which is the same as
$$|\{l\in [0,m/d-1]: (l,m/d)=1\ {\rm and\ } l\not\equiv \ktup\ \mod{p}\text{ for all }\ p|m/d\}|$$
$$= |\{l\in[0,m/d-1]: l\not\equiv 0,\ktup\ \mod{p}\text{ for all }\ p|m/d\}|.$$
Defining
$$\rpa:=|\{s\in [0,p-1]: b_i \equiv s \mod{p}\ {\rm for\ some\ } i=1,\dots, k\}|$$
we see that estimating (\ref{needaname}) is equivalent to estimating
$$\prod_{p|m/d}(p-\rpa-1) = \frac{m}{d}\prod_{p|m/d}\left(1-\frac{\rpa+1}{p}\right).$$

Hence, we have shown that 
\begin{equation}\label{S}S_d\leq \sum_{\ktup\in B} \frac{m}{d} \prod_{p|m/d} \left( 1-\frac{\rpa +1}{p}\right).
\end{equation}
In order to bound this sum from above we need to understand the function $\rpa$.   We notice that if $p$ is much larger than $k$, then a random $k$-tuple will intersect  $k$ distinct residue classes$\pmod{p}$ with high probability, and so $\rpa$ is typically of size $k$.  The following lemma quantifies this fact.

\begin{lemma} \label{nonuniform ktuples} For $\ktup\in B$ let
$$f(\ktup)= \sum_{p | m \atop \rpa \leq k-1} {1 \over p}.$$
Then there is an absolute constant $c>0$ such that for every $\beta\in\mathbb{R}^+$ we have
$$\left|\left\{ \ktup\in B: f(\ktup) \geq \beta\right\}\right|\leq k^2 2^{- \exp(\beta/ck^2)} |B|^{k-2} \varphi(m)^2.$$
\end{lemma}
\begin{proof}
The result will follow from optimizing a double-counting argument on the quantity
$$\sum_{b_i,b_j\in B} \left(\sum_{p|b_i-b_j \atop p|m}\frac{1}{p}\right)^l$$
over positive integer values of $l$.

\smallskip

\noindent
{\bf Upper Bound: } We expand out the exponent and rearrange summation to get
$$\sum_{p_1,\dots, p_l | m} \frac{1}{p_1\cdots p_l}\sum_{b_i,b_j\in B\atop b_i\equiv b_j\ \mod{\lcm(p_1,\dots,p_l)}} 1.$$
Now, fixing the $l$-tuple $\tup{p}{l}$, we suppose that the distinct primes among this $l$-tuple are $\tup{p}{u}$.  Then the inner sum above is bounded above by
$$\sum_{x,y\in\mathbb{Z}_m^\ast \atop x\equiv y\ \mod{\lcm(p_1,\dots,p_l)}} 1 = \frac{\varphi(m)^2}{\varphi(\lcm(p_1,\dots,p_l))}.$$
It follows that the original quantity is bounded from above by
\begin{eqnarray}
\varphi(m)^2 \sum_{p_1,\cdots,p_l | m}{1 \over p_1 \dots p_l \varphi(\lcm(p_1,\dots,p_l))}&\leq&
\varphi(m)^2 \left ( \sum_{p |m} {1 \over p (p-1)^{1/l}} \right )^l.\nonumber
\end{eqnarray}
Splitting the remaining sum based on whether $p$ is greater than or less than $l^l$ and analyzing appropriately we see that the inner sum over $p|m$ is smaller than
\begin{eqnarray}
\sum_{2 \leq p \leq l^l} {1 \over p}\ + \sum_{n \geq l^l} {1 \over n^{1+1/l}}
& \ll & \log(l) + \int_{l^l}^\infty {dx \over x^{1+1/l}} \nonumber\\
& \ll & \log(l).
\end{eqnarray}
It follows that
$$\sum_{b_i,b_j\in B} \left(\sum_{p|b_i-b_j \atop p|m}\frac{1}{p}\right)^l\leq \varphi(m)^2 (c \log(l))^l,$$
where $c > 0$ is some constant.

\smallskip

\noindent
{\bf Lower Bound: } Let
$$K=K(\beta)=\left\{\ktup\in B: \sum_{p | m \atop \rpa \leq k-1} {1 \over p}\ \geq\ \beta\right\}.$$
Given $\ktup\in K$, we have
$$\sum_{i,j=1,\dots,k\atop i\neq j} \sum_{p|m \atop p|b_i-b_j} \frac{1}{p} \geq \sum_{p|m\atop p|b_i-b_j\ {\rm for\ some\ } i\neq j} \frac{1}{p} = \sum_{p | m \atop \rpa \leq k-1} {1 \over p}\ \geq\ \beta$$
so there must be some pair, $b_i,b_j$, such that
$$\sum_{p | m \atop p | b_i-b_j} {1 \over p}\ \geq\ {\beta \over
{k \choose 2}}.$$
At least one such pair comes from each $k$-tuple $\ktup\in K$, and a given pair $b_i,b_j$ can appear in at most $k^2|B|^{k-2}$ $k$-tuples.  We therefore have
$$\sum_{b_i,b_j\in B} \left(\sum_{p|b_i-b_j \atop p|m}\frac{1}{p}\right)^l\geq {|K| \over k^2 |B|^{k-2}} \beta^l {k \choose 2}^{-l}.
$$

\bigskip

\noindent
Combining the upper and lower bounds gives
$${|K| \over k^2 |B|^{k-2}} \beta^l {k \choose 2}^{-l}\ \leq\
c^l (\log l)^l  \varphi(m)^2.$$
It follows that
$$|K|\ \leq\ \beta^{-l} {k \choose 2}^l
k^2 c^l (\log l)^l |B|^{k-2} \varphi(m)^2.$$
Taking
$$l\ =\ \exp( \beta / c k^2)$$
we find that
\[|K|\ \leq\ k^2 2^{-l}|B|^{k-2} \varphi(m)^2.\qedhere\]
\end{proof}

We are now ready to proceed with the proof of Proposition \ref{representations}.  We would like to break the sum $S_d$ into pieces defined by the behavior of the function $\rpa$.  For technical reasons, we must first define some additional notation.   Define
$$m_1:=\prod_{p|m \atop p\leq 3k} p\ \ \ \ \ \ \ \ \ \ \ m_2:= \prod_{p|m \atop p>3k} p,$$
and for any $d|m$ define
$$d_1 := \prod_{p|d \atop p\leq 3k} p\ \ \ \ \ \ \ \ \ \ d_2 := \prod_{p|d \atop p> 3k} p.$$  Then we can rewrite (\ref{S}) to have 
\begin{eqnarray}
S_d & \leq & \sum_{\ktup \in B}{m_1\over d_1}\prod_{p|m_1/d_1}\left(1-{1\over p}\right){m_2\over d_2}\prod_{p|m_2/d_2}\left( 1-{\rpa +1 \over p}\right)\nonumber\\
& = & \varphi(m_1/d_1) \sum_{\ktup \in B}{m_2\over d_2}\prod_{p|m_2/d_2}\left( 1-{\rpa +1 \over p}\right).\nonumber
\end{eqnarray}
For a fixed $d$ and a $\ktup\in B$, let
$$P = \{p|m_2/d_2 : \rpa \leq k-1\}.$$
Expanding $S_d$ as a sum over geometric intervals we have
\begin{eqnarray}
S_d & \leq & \varphi(m_1/d_1) \sum_{j=-\infty}^{\infty} \sum_{\ktup\in B \atop f(\ktup)\in [2^j,2^{j+1}) } {m_2\over d_2}\prod_{p|m_2/d_2}\left( 1-{\rpa +1 \over p}\right)\nonumber\\
& \leq & \varphi(m_1/d_1) \sum_{j=-\infty}^{\infty} \sum_{\ktup\in B \atop f(\ktup)\in [2^j,2^{j+1}) } {m_2\over d_2} \prod_{p|m_2/d_2} \left( 1-{k +1 \over p}\right)\prod_{p\in P} \left( 1-{k +1 \over p}\right)^{-1}.\nonumber
\end{eqnarray}
Now, for fixed $j$ and $\ktup\in B$ we have
\begin{eqnarray}
\log\left(\prod_{p\in P} \left( 1-{k +1 \over p}\right)^{-1}\right) & = & -\sum_{p\in P} \log\left(1-\frac{k+1}{p}\right)\nonumber\\
& = & \sum_{p\in P} \frac{k+1}{p}\left( 1+ \frac{1}{2}\left(\frac{k+1}{p}\right) + \frac{1}{3}\left( \frac{k+1}{p}\right)^2+\dots\right)\nonumber\\
& \leq & (k+1)2^{j+1},
\end{eqnarray}
where the second to last line follows from the fact that all primes in $P$ are larger than $3k$.  
Applying Lemma \ref{nonuniform ktuples} we now have
\begin{eqnarray}\label{Sd}
S_d & \leq & \varphi(m_1/d_1) {m_2\over d_2} \prod_{p|m_2/d_2} \left( 1-{k +1 \over p}\right)\sum_{j=-\infty}^{\infty} \sum_{\ktup\in B \atop f(\ktup)\in [2^j,2^{j+1}) } e^{2(k+1)2^j}\nonumber\\
& \leq & k^2 \varphi(m_1/d_1) {m_2\over d_2} |B|^{k-2}\varphi(m)^2\prod_{p|m_2/d_2} \left( 1-{k +1 \over p}\right)\sum_{j=0}^{\infty}e^{2(k+1)2^j}2^{-\exp(2^j/ck^2)}\nonumber\\
& = & k^2C_k\varphi(m_1/d_1)\frac{m_2}{d_2}|B|^{k-2}\varphi(m)^2\prod_{p|m_2/d_2} \left( 1-{k +1 \over p}\right)
\end{eqnarray}
where 
\begin{equation}\label{C}C_k=\sum_{j=0}^{\infty}e^{2(k+1)2^j}2^{-\exp(2^j/ck^2)}.
\end{equation}
Note that the lower bound on the range of summation in $j$ comes from the fact that for every $k$-tuple $\ktup\in B$ we have $f(\ktup)\geq \sum_{p\leq k} 1/p \gg \log\log k>2 $ provided $k$ is large enough.  Note also that the sum is clearly convergent to some constant dependent only on $k$.

We also note that since we will see that $C_k$ has size $e^{Ck^3\log k}$, in the argument below we have absorbed several smaller functions of $k$ into $C_k$.

The remainder of the proof follows in two steps.  We first find a bound for $S$ in terms of $C_k$ by summing $\sum_{d|m}S_d$ and then we compute an upper bound for $C_k$.  In particular, we will show 

\begin{equation}\label{Sfinal} S\leq \frac{C_k^3|B|^k\varphi(m)^k}{\alpha^2m^{k-1}}
\end{equation}
and 

\begin{equation}\label{Cfinal} C_k \leq e^{Ck^3\log(k)}.
\end{equation}

We start by proving (\ref{Sfinal}).  Summing (\ref{Sd}) over all $d|m$, we have
$$S\leq k^2C_k |B|^{k-2} \varphi(m)^2 \sum_{d|m} \varphi(m_1/d_1) {m_2\over d_2}\prod_{p|m_2/d_2} \left( 1-{k +1 \over p}\right).$$
Noticing that $\varphi(m_1/d_1)\leq m_1$ we have $\varphi(m_1/d_1)m_2\leq m$.  Further, we observe that the number of divisors $d$ of $m$ which will give the same $d_2$ is the number of ways of choosing which primes $p\leq 3k$ will be factors of $d$.  In other words, there are fewer than
$$\sum_{t=0}^{3k} \binom{3k}{t} = 2^{3k}$$
such values of $d$.  We can now rewrite the bound on $S$ as
\begin{eqnarray}
S & \leq & C_k^2 |B|^{k-2} \varphi(m)^2 m \sum_{d_2|m_2} {1\over d_2}\prod_{p|m_2/d_2} \left( 1-{k +1 \over p}\right)\nonumber\\
& = & C_k^2\alpha^{-2} |B|^k m \prod_{p | m_2} \left ( 1 - {k+1 \over p} \right ) \sum_{d_2 | m_2} {1 \over d_2 \prod_{p | d_2} (1 - (k+1)/p)}\nonumber\\
& = & C_k^2\alpha^{-2} |B|^k m \prod_{p | m_2} \left ( 1 - {k+1 \over p} \right )\left ( 1 + {1 \over p - k - 1} \right )\nonumber\\
& \leq & C_k^2\alpha^{-2} |B|^k m \prod_{p | m_2} \left( 1-{1\over p}\right)^k
\end{eqnarray}
where the last step follows by recalling that primes dividing $m_2$ are larger than $k$.
Using the fact that $|B| = \alpha \varphi(m)$ and the fact that
$\varphi(m) = m \prod_{p|m} (1-1/p)$, the bound becomes
$$S \leq \frac{C_k^2}{\alpha^2} \frac{|B|^k\varphi(m)^k}{m^{k-1}}\prod_{p|m_1}\left(1-\frac{1}{p}\right)^{-k}.$$
The remaining product is less than $(3k)^k$, which is smaller than $C_k$.

It remains to prove (\ref{Cfinal}).    Recall that 
\begin{equation}\label{Ck}
C_k=\sum_{j=0}^{\infty}e^{(k+1)2^{j+1}-\exp(2^j/ck^2)\log(2)}
\end{equation}
Expanding out the exponential function in the exponent, we see that the entire exponent is smaller than
$$2(k+1)2^{j+1}-\log(2)-\frac{2^j}{ck^2}-\frac{2^{2j}}{2c^2k^4}.$$
We notice that if $j$ is larger than $\log_2(c^2k^4(k+1)),$ then 
$$\frac{2^{2j}}{2c^2k^4} \geq (k+1)2^{j+1}.$$
Hence for $j\geq\log_2(4c^2k^4(k+1))$, the exponent is smaller than $\log(2)-2^j/ck^2$ and so the tail of the sum is bounded by $e^{1/(ck^2)}$.  Furthermore, we find that  for small $j$, the exponent is maximized when 
$$2^j = ck^3\log(4ck^2(k+1)/\log 2).$$
Hence, splitting the sum in (\ref{Ck}), we have 
$$C_k\leq \log_2(4c^2k^4(k+1))e^{ck^3\log(4ck^2(k+1)/\log 2)}+e^{1/ck^2}.$$  Inequality (\ref{Cfinal}) follows.  

\noindent
Combining (\ref{Sfinal}) and (\ref{Cfinal}), we have
$$S \leq \frac{e^{Ck^3\log(k)}}{\alpha^2} \frac{|B|^k\varphi(m)^k}{m^{k-1}}$$
as desired.\qed


\section{Completion of the Proof of Theorem \ref{sumsetsinPN}}


In this section we complete the proof of Theorem \ref{sumsetsinPN}, which implies Theorem \ref{sumsetsinP}, the main result of this paper.  

Let $n$ be sufficiently large, let $A_n\subset \mathcal{P}_n$ satisfy $|A_n|\geq \delta |\mathcal{P}_n|$, and let $\epsilon > 0$.  Suppose that $G$ is as  in Section 3, and let $\alpha$ be such that $|G|=\alpha\varphi(m)$ (in particular, $\alpha\geq \delta/2$).  Then, by Proposition \ref{sumset1}, for every $W$ sufficiently large in terms of $\delta$ and $\epsilon$ we have
\begin{equation}\label{sumset}
|A_n + A_n| \geq \sum_{x\in G+G} \left(\Delta_x-\epsilon \right) \frac{n}{m}
\end{equation}
where $m=\prod_{p\leq W} p$.
Since $$\sum_{b\in\mathbb{Z}_m^{\ast}}\delta_b\geq \delta\varphi(m),$$ we can show that 
$$\sum_{b\in G} \delta_b \geq \frac{\delta}{2}\varphi(m).$$

Hence, we also see that
\begin{equation}\label{last}\sum_{(b,b')\in G\times G} \left(\frac{\delta_b+\delta_{b'}}{2}\right) \geq \frac{\delta}{2\alpha}|G|^2.
\end{equation}
Setting
$$r(x)=|\{(b,b')\in G\times G: b+b'=x\}|$$
we can see that (\ref{last}) is equivalent to
$$\sum_{x\in G+G} r(x)\gamma_x \geq \frac{\delta}{2\alpha}|G|^2$$
where
$$\gamma_x = \frac{1}{r(x)}\sum_{(b,b'):\ b+b'=x} \left(\frac{\delta_{b}+\delta_{b'}}{2}\right)\leq \Delta_x.$$
Using H\"older's inequality, we have
$$\left(\sum_{x\in G+G} r(x)^k\right)^{\frac{1}{k}}\left(\sum_{x\in G+G} \Delta_x^{\frac{k}{k-1}}\right)^{\frac{k-1}{k}}\geq \frac{\delta}{2\alpha}|G|^2.$$
Applying Proposition \ref{representations} we find
$$\left(\sum_{x\in G+G} \Delta_x^{\frac{k}{k-1}}\right)^{\frac{k-1}{k}}\geq \frac{\delta}{2}\left(\frac{\alpha^{2/k}}{e^{Ck^2\log(k)}}\right) m^{(k-1)/k}.$$
Because $k/(k-1)>1$ and $\Delta_x\leq 1$, we have
$$\sum_{x\in G+G} \Delta_x \geq \left(\frac{\delta}{2}\right)^{k/(k-1)}\alpha^{2/(k-1)}e^{-C\frac{k^3\log(k)}{k-1}} m.$$

Proceeding as in the proof of Theorem \ref{ZNstar} and then returning to (\ref{sumset}), we have
$$|A_n+A_n|\geq \left(C_1\delta e^{-C_2(\log(1/\delta))^{2/3}(\log\log(1/\delta))^{1/3}}-\epsilon\right) n$$
for every $\epsilon>0$.


\section{Acknowledgements}


We are extremely grateful to Ernie Croot for suggestions in proving Theorem \ref{ZNstar}.  We thank Izabella \L aba, Neil Lyall and Akos Magyar for their support and advice.  
\bibliographystyle{plain}

\vspace{.2 in}
  \textsc{Department of Mathematics, University of British Columbia, Vancouver, BC V6T 1Z2}
  
  \textit{E-mail address:} karstenc@math.ubc.ca\\
  
  \textsc{Department of Mathematics, University of Georgia, Athens, GA 30602}
  
  \textit{E-mail address:} mhamel@math.uga.edu\\

\end{document}